\documentclass[a4paper]{amsart}

\usepackage{amssymb}
\usepackage{amsrefs}
\usepackage{multirow}
\usepackage{hyperref}

\BibSpec{article}{%
  +{}{\PrintAuthors}  		{author}
  +{,}{ \textit}     		{title}
  +{,}{ }             		{journal}
  +{}{ \textbf}       		{volume}
  +{}{ \parenthesize} 		{date}
  +{,}{ }      	      		{conference}
  +{,}{ }      	      		{book}
  +{,}{ }            		{pages}
  +{,}{ }            	 	{note}
  +{,}{ }            	 	{status}
  +{, available at}{  \texttt } {eprint}
  +{.}{}              {transition}
}
\BibSpec{book}{%
  +{}{\PrintAuthors}  {author}
  +{,}{ \textit}      {title}
  +{,}{ }	      {note}
  +{,}{ }             {publisher}
  +{,}{ }             {place}
  +{,}{ }             {date}
  +{.}{}              {transition}
}

\newtheorem{theorem}{Theorem}[section]
\newtheorem*{theorem*}{Theorem}

\theoremstyle{plain}
\newtheorem{corollary}[theorem]{Corollary}
\newtheorem{lemma}[theorem]{Lemma}
\newtheorem{proposition}[theorem]{Proposition}

\newtheorem*{lemma*}{Lemma}
\newtheorem*{proposition*}{Proposition}

\theoremstyle{definition}

\newcommand{\CC}{\mathbb{C}}

\newcommand{\RR}{\mathbb{R}}

\newcommand{\calH}{\mathcal{H}}
\newcommand{\calK}{\mathcal{K}}

\newcommand{\calM}{\mathcal{M}}

\newcommand{\dt}{\left.\frac{d}{dt}\right|_{t=0}}

\DeclareMathOperator{\Val}{Val}
\DeclareMathOperator{\vol}{vol}

\DeclareMathOperator{\Grass}{Gr}
\DeclareMathOperator{\Sym}{Sym}

\DeclareMathOperator{\Hess}{Hess}
\DeclareMathOperator{\tr}{tr}
\DeclareMathOperator{\Klain}{Kl}

\begin{document}

\title[Aleksandrov-Fenchel inequalities for unitary valuations]{Aleksandrov-Fenchel inequalities for unitary valuations of degree $2$ and $3$}

\author{Judit Abardia} 
\address[Judit Abardia]{	Institut f\"ur Mathematik\\
				Goethe-Universit\"at Frankfurt\\
				Robert-Mayer-Str. 10\\
				60325 Frankfurt am Main\\
				Germany}
\email{abardia@math.uni-frankfurt.de}

\author{Thomas Wannerer}
\address[Thomas Wannerer]{	Institut f\"ur Mathematik\\
				Goethe-Universit\"at Frankfurt\\
				Robert-Mayer-Str. 10\\
				60325 Frankfurt am Main\\
				Germany}
\email{wannerer@math.uni-frankfurt.de}

\thanks{The first-named author was supported by DFG grant  BE 2484/3-1 and the second-named author was supported by DFG grant BE 2484/5-1}

\begin{abstract}We extend the classical Aleksandrov-Fenchel inequality for mixed volumes to functionals arising naturally in hermitian integral geometry. As a consequence, we obtain 
 Brunn-Minkowski and isoperimetric inequalities for hermitian quermassintegrals.
\end{abstract}

\maketitle

\section{Introduction}

The Aleksandrov-Fenchel inequality for mixed volumes states that
\begin{equation}\label{eq:aleksandrov_fenchel}V(K_1,K_2,K_3, \ldots, K_n)^2 \geq  V(K_1,K_1,K_3, \ldots K_n) V(K_2, K_2, K_3,\ldots, K_n)\end{equation}
for all convex bodies $K_1,K_2,\ldots, K_n$ in $\RR^n$ ($n\geq 2$). A whole series of important inequalities between mixed volumes of convex bodies, including the
Brunn-Minkowski and isoperimetric inequalities for quermassintegrals, can be deduced from \eqref{eq:aleksandrov_fenchel} and hence the Aleksandrov-Fenchel inequality can be regarded as the 
main inequality in the Brunn-Minkowski theory of convex bodies. Special cases of \eqref{eq:aleksandrov_fenchel} have been extended to non-convex domains, see \cites{chang_wang13, guan_li09, michael_simon73}.
For applications of the Aleksandrov-Fenchel inequality to the geometry of convex bodies and other fields such as combinatorics, geometric analysis and mathematical physics, 
we refer the reader to \cites{andrews96,chang_wang11, lam10, lutwak85, lutwak93, schneider14, stanley81}
and the references therein. 

Several different proofs of the Aleksandrov-Fenchel inequality are known. In $\RR^3$, the first proof of \eqref{eq:aleksandrov_fenchel} was discovered by  Minkowski \cite{minkowski03} in 1903. In the 1930s, 
Aleksandrov \cites{aleksandrov37,aleksandrov38} gave two different proofs of his inequality, 
one based on strongly isomorphic polytopes and another, building on ideas of Hilbert \cite{hilbert12}*{Chapter 19}, based on elliptic operator theory. 
Around the same time, also Fenchel \cite{fenchel36}  sketched a proof of the inequality \eqref{eq:aleksandrov_fenchel}. In the 1970s,  Khovanski\u\i\ \cite{burago_zalgaller88}*{Section 27}
 and Teissier \cite{teissier79} independently discovered that
the Aleksandrov-Fenchel inequality can be deduced from the Hodge index theorem from algebraic geometry.
More recently, special cases of \eqref{eq:aleksandrov_fenchel} have been proved using optimal mass transport and curvature flow techniques, see \cites{alesker_etal99, chang_wang13, guan_li09}.
For a more complete account of the history of the Aleksandrov-Fenchel inequality, we refer the reader to
 \cite{schneider14}*{p. 398}.

In this work we extend the Aleksandrov-Fenchel inequality for mixed volumes to functionals arising naturally in hermitian integral geometry \cites{alesker03, bernig_fu11}. One way to describe these functionals is as follows:
It is a well-known fact that for $1<k<2n-1$  the action of the unitary group $U(n)$ decomposes the Grassmannian $\Grass_k=\Grass_k(\CC^n)$ 
of $k$-dimensional, real subspaces of $\CC^n$ into infinitely many orbits. 
For $k=2,3$ and $n\geq k$ the orbits of $\Grass_k(\CC^n)$  can be described  by a single real parameter, known as the K\"ahler angle 
$\theta\in [0,\pi/2]$. For example, isotropic (with respect to the standard K\"ahler form on $\CC^n$) subspaces have K\"ahler angle $\pi/2$ and complex subspaces have K\"ahler angle $0$. 
For each  K\"ahler angle  we define two functionals on $\calK(\CC^n)$, the space of convex bodies, i.e.\ non-empty, compact convex sets, of $\CC^n$,
$$\varphi_\theta(K)= \int_{\Grass_2(\theta)} \vol_2(K|E)\;  dE$$
and, for $n\geq 3$,
$$ \psi_\theta(K)= \int_{\Grass_{3}(\theta)} \vol_3(K|E)\;  dE.$$
Here $\Grass_k(\theta)$ ($k=2,3$) denotes the orbit of $\Grass_k(\CC^n)$ corresponding to the  K\"ahler angle $\theta$, $\vol_k(K|E)$ is the $k$-dimensional volume of the orthogonal
projection of the convex body $K$ on the $k$-dimensional subspace $E$,  and $dE$ denotes the $U(n)$-invariant probability measure on the orbit. 

Any linear combination $\mu$  of the functionals $\varphi_\theta$ (respectively,  $\psi_\theta$) is called a unitary valuation. Observe that $\mu(t K)=t^2 \mu(K) $ (respectively, $\mu(t K)=t^3\mu(K)$) for $t>0$.  
If $\mu$ is homogeneous of degree $k$, then 
$$\mu(K_1,K_2,\ldots, K_k)= \frac{1}{k!} \left.\frac{\partial^k}{\partial t_1 \partial t_2\cdots \partial t_k}\right|_{0} \mu(t_1K_1 + t_2K_2+\cdots +t_kK_k)$$
is called the polarization of $\mu$. Here $K+L$ denotes the Minkowski sum of convex bodies. Note that $\mu(K,K,\ldots,K)=\mu(K)$.

Our main result is as follows:

\begin{theorem}\label{th1}
 If $\mu$ belongs to the convex cone generated by $\psi_\theta$ with
\begin{equation}\label{eq:theta3} 0\leq \cos^2 \theta\leq \frac{3(n+1)}{5n-1},\end{equation}
then
\begin{equation}\label{eq:AF3}\mu(K,L,M)^2 \geq \mu(K,K,M) \mu(L,L,M)\end{equation}
for all convex bodies $K,L,M$. Moreover, if $\mu$ belongs to the convex cone generated by $\varphi_\theta$  with
\begin{equation}\label{eq:theta2} 0\leq \cos^2\theta\leq \frac{n+1}{2n},\end{equation}
then
\begin{equation}\label{eq:AF2}\mu(K,L)^2 \geq \mu(K,K) \mu(L,L).\end{equation}
If $\mu=\varphi_\theta$ with 
$\frac{n+1}{2n}<\cos^2\theta$,
then there exist convex bodies for which \eqref{eq:AF2} does not hold.  
\end{theorem}

We remark that as in the classical Aleksandrov-Fenchel inequality \eqref{eq:aleksandrov_fenchel} our inequalities also hold if the first convex body is replaced by the difference of support functions of two convex bodies.
Moreover, since $\theta$ can be chosen such that $\psi_\theta(K)$ is proportional to
$$V(K,K,K,B,\ldots,B),$$
where $B$ denotes the unit ball in $\CC^n$, see Lemma~\ref{lem:quermassintegrals} below, the inequality \eqref{eq:AF3} contains the Aleksandrov-Fenchel inequality with $K_4=\cdots=K_{2n}=B$ as a special case.

Aleksandrov's second proof \cite{aleksandrov38} of \eqref{eq:aleksandrov_fenchel}, which we follow closely, makes critical use of Aleksandrov's inequality for mixed discriminants. 
To prove \eqref{eq:AF3}, we first use G\aa rding's theory of hyperbolic
polynomials \cite{garding59} to establish a hermitian analog of this fundamental determinantal inequality (see Proposition~\ref{prop:hyperbolicity} below) and then associate to each $\mu$ an elliptic differential operator.

A complete characterization of the equality cases in the Aleksandrov-Fenchel inequality 
\eqref{eq:aleksandrov_fenchel} is not known. However, in various special cases such a characterization exists, see \cite{schneider14}*{Section 7.6}. 
We say that a convex body $K$ is  $C^{2,\alpha}_+$ if its support function $h_K$ lies in the H\"older space $C^{2,\alpha}(S^{n-1})$ and $\det( \overline\nabla^2h_K + h_K \overline g)>0$,
where $\overline g$ denotes the standard Riemannian metric on the unit sphere $S^{n-1}$ and $\overline \nabla$ denotes the covariant derivative with respect to this metric.
We denote by $\calH_1$ the space of spherical harmonics of degree one, i.e.~restrictions of linear functionals to the unit sphere, and by $\calH_{1,1}\subset\calH_2$ the subspace of spherical harmonics of degree $2$ which 
are invariant under the canonical circle action on the odd-dimensional sphere $S^{2n-1}\subset \CC^n$.
 We establish the following description of equality cases in the inequalities \eqref{eq:AF3} and \eqref{eq:AF2}.

\begin{theorem}\label{th2}
   Suppose  $\mu$ belongs to the convex cone generated by $\psi_\theta$ with 
\begin{equation} \label{eq:theta3_open} 0< \cos^2 \theta< \frac{3(n+1)}{5n-1}\end{equation}
 and $M\in C^{2,\alpha}_+$. If $\mu(L,L,M)>0$,
then equality
holds in the inequality
$$\mu(K,L,M)^2 \geq \mu(K,K,M)\mu(L,L ,M)$$
if and only if $K$ and $L$ are homothetic. If $M$ is a ball, then the above characterization extends to $\cos^2 \theta=0$ and $\frac{3(n+1)}{5n-1}$.

If $\mu$ belongs to the convex cone generated by $\varphi_\theta$ with 
\begin{equation}\label{eq:theta2_halfopen}0\leq \cos^2\theta< \frac{n+1}{2n}\end{equation}
and $\mu(L,L)>0$,
then equality holds in the inequality
$$\mu(K,L)^2 \geq \mu(K,K)\mu(L,L)$$
if and only if $K$ and $L$ are homothetic.
If $\mu= \varphi_{\theta}$ with  $\cos^2\theta= \frac{n+1}{2n}$, then equality holds if and only if
 there exists a constant $\alpha$ such that $ h_K$ and $ \alpha h_L$ differ by an element of $\calH_{1} \oplus \calH_{1,1}$. 
\end{theorem}

We remark that we obtain the above characterization of equality cases in the more general situation where $K$ is replaced by the difference of support functions of two convex bodies.

As a consequence of Theorems~\ref{th1} and \ref{th2}, we obtain among several other inequalities a hermitian extension of the Brunn-Minkowski inequality (see Theorem~\ref{thm:BM} below) 
and the following isoperimetric inequalities for hermitian quermassintegrals.

\begin{theorem}\label{th3} Let $\theta\in [\pi/4, \pi/2]$ and choose $\theta'\in [0,\pi/2]$ such that 
$3 \cos^2\theta' = \cos^2 \theta$.
Then 
$$\left(\int_{\Grass_1} \vol_1(K|E) \; dE \right)^2 \geq \frac{4}{\pi} \int_{\Grass_2(\theta)} \vol_2(K|E) \; dE $$
and 
$$\left( \int_{\Grass_2(\theta')} \vol_2(K|E) \; dE \right)^3 \geq \frac{9\pi}{16} \left( \int_{\Grass_3(\theta)} \vol_3(K|E) \; dE \right)^2$$
for all convex bodies in $\CC^n$. If the left-hand side is non-zero, then equality holds if and only if $K$ is a ball.

\end{theorem}

The above inequalities hold in particular for averages over isotropic (resp.\ Lagrangian) subspaces ($\theta=\pi/2$), but the case of complex subspaces ($\theta=0$) is not covered by the theorem.
In fact, although the inequalities hold for a slightly larger range of $\theta$ than 
stated in Theorem~\ref{th3}, we show in Proposition~\ref{prop:counterexample} that the first inequality fails for $\theta<\pi/4$ when $n$ is sufficiently large.

\section{Valuations and area measures}

Valuations are a classical notion from convex geometry. A function $\mu\colon \calK(V)\to \RR$ on the set of non-empty, convex, compact subsets of a finite-dimensional vector space is called a valuation if 
$$\mu(K\cup L)= \mu(K) + \mu(L)-\mu(K\cap L)$$
whenever the union of $K$ and $L$ is again convex.  The space of continuous (with respect to the Hausdorff metric) and translation-invariant
valuations is denoted by $\Val= \Val(V)$. A valuation  $\mu$  is called
 homogeneous of degree $k$ if $\mu(\lambda K)=\lambda^k \mu(K)$ for every $\lambda >0$ and $\Val_k\subset \Val$ denotes the subspace of $k$-homogeneous valuations.
 By a fundamental result of McMullen \cite{mcmullen77}, every continuous and translation-invariant
 valuation is the sum of homogeneous valuations
$$\Val=\bigoplus_{k=0}^n \Val_k.$$
As a consequence, one can associate to each $\mu\in \Val_k$  a unique function on the $k$-fold product
 $\calK(V)\times\cdots \times \calK(V)$, which is again denoted by $\mu$ and called the polarization of $\mu$, such that (i) $\mu(K,\ldots, K)=\mu(K)$; (ii) $\mu$ is symmetric in its arguments; and (iii)
for every $K,L, K_2,\ldots,K_k\in \calK(V)$ and $ s,t>0$
$$\mu( sK+tL,K_2\ldots, K_{k}) =s \mu( K,K_2\ldots, K_{k}) + t\mu(L,K_2\ldots, K_{k}) .$$
If $P$ is just a point, then, by the translation-invariance of $\mu$, 
\begin{equation}\label{eq:val_kernel}\mu(P, K_2,\ldots, K_k)=0.\end{equation}

From now on let $V$ be a finite-dimensional, euclidean vector space. The support function of $K\in\calK(V)$ is the function on the unit sphere of $V$ 
defined by $h_K(u)= \sup_{x\in K} \left\langle u,x\right\rangle $, where $\left\langle u,x\right\rangle$ denotes the inner product on $V$. If $f$ is the difference of two support functions, say $f=h_K-h_L$, then one 
defines 
$$\mu(f,K_2,\ldots, K_{k})= \mu(K,K_2,\ldots, K_{k})-\mu(L,K_2,\ldots, K_{k}).$$
 Similarly,   $\mu( f_1, f_2, K_3,\ldots, K_{k})$, where $f_1$ and $f_2$ are differences of support functions, is defined. In the following we will make frequent use of the fact that every $C^2$ function on the sphere
is the difference of two support functions, see, e.g., \cite{schneider14}*{Lemma 1.7.8}.

We denote by $\Grass_k=\Grass_k(V)$ the Grassmannian of real $k$-dimensional subspaces of $V$. If $\mu\in \Val_k$ is even, that is $\mu(-K)=\mu(K)$, then, by a theorem of Hadwiger (see below), the restriction of $\mu$ to 
$E\in \Grass_k$ is a multiple of the $k$-dimensional Lebesgue measure on $E$, and the corresponding factor is denoted by $\Klain_\mu(E)$. The function $\Klain_\mu\colon \Grass_k\to \RR$ is called the Klain function of
$\mu$ and, by a theorem of Klain \cite{klain00}, it determines $\mu$ uniquely.

A celebrated theorem of Hadwiger characterizes linear combinations of the intrinsic volumes, which are defined by
$$\mu_k(K)= \frac{1}{\omega_{n-k}} \binom{n}{k} V(\underbrace{K,\ldots,K}_{k \text{ times}},B,\ldots,B),$$
where $B\subset \RR^n$ is the euclidean unit ball and $\omega_{k}$ is the volume of the $k$-dimensional euclidean unit ball,
 as the only valuations on $\RR^n$ which are continuous and isometry invariant. 
In particular, this result shows that the space of continuous and isometry invariant valuations on $\RR^n$ is finite-dimensional. 

In \cite{alesker03} Alesker proved the following hermitian extension of Hadwiger's theorem: the space of valuations on $\CC^n$ which are continuous and invariant under affine unitary 
transformations is finite-dimensional. The space of these valuations is denoted by $\Val^{U(n)}$ and its elements are called unitary valuations. Here $U(n)$ denotes the group of unitary transformations, i.e.\ those
 $\CC$-linear maps $A\colon\CC^n\rightarrow \CC^n$ which preserve the 
standard K\"ahler form $\omega  = \sum_{i=1}^n dx_i \wedge dy_i$.  
For every integer $0\leq k\leq 2n$, we denote by $\Val_k^{U(n)}$ the subspace of $k$-homogeneous valuations. While $\Val^{SO(n)}_k$ is one-dimensional and spanned by the intrinsic volume $\mu_k$, Alesker proved in \cite{alesker03} that 
\begin{equation}\label{eq:dimension}\dim \Val^{U(n)}_k = 1 + \min\left\{ \left\lfloor \frac{ k}{2} \right\rfloor , \left\lfloor \frac{2n- k}{2} \right\rfloor \right\}.\end{equation}                                                                                                                                                                       
For more information on valuation theory,  see \cites{alesker01,alesker_etal11,alesker_faifman,haberl_parapatits14,haberl_parapatits15,ludwig06,ludwig10,ludwig_reitzner10,parapatits_schuster12,parapatits_wannerer13,schuster08,schuster10} and the references therein. For recent applications to 
integral geometry, we refer the reader to \cites{abardia_etal12,alesker10,bernig09a,bernig09b,bernig_fu06,bernig_fu11,bernig_etal,fu06,wannerer13+}.

In this article we establish inequalities for unitary valuations of degree $2$ and $3$. Let us describe the spaces $\Val^{U(n)}_k$ for $k=2,3$ and $ n\geq k$ explicitly.
For $k=2,3$, the action of $U(n)$  decomposes $\Grass_k(\CC^n)$  into 
infinitely many orbits parametrized by the K\"ahler angle $\theta\in [0, \pi/2]$.  Given $E\in \Grass_k$, the K\"ahler angle $\theta=\theta(E)$  is defined by
$$\cos^2\theta = | \omega_E|^2,$$
where $\omega_E$ denotes the restriction of the K\"ahler form $\omega$ to $E$ and $| \cdot|$ denotes the induced euclidean norm
on $\wedge^2 E$.  The K\"ahler angle of $E\in \Grass_{2n-k}(\CC^n)$ is, by definition,
the K\"ahler angle of $E^\perp$. In a similar way,
 the $U(n)$-orbits of $\Grass_k(\CC^n)$ for $3<k<2n-3$ can be described by multiple K\"ahler angles, see \cite{tasaki03}.

Since every $\mu\in \Val_k^{U(n)}$ is even, it is uniquely determined by its Klain function $\Klain_\mu$, which, by the $U(n)$-invariance of $\mu$,  is constant on every $U(n)$-orbit. 
For $k=2,3$ and $n\geq k$, the space $\Val_k^{U(n)}$ is $2$-dimensional and 
spanned by two special valuations $\mu_{k,0}$ and $\mu_{k,1}$. In terms of Klain functions, they are given by
$$\Klain_{\mu_{k,0}}= 1-\cos^2 \theta \qquad \text{and}\qquad \Klain_{\mu_{k,1}} =\cos^2 \theta$$
where $\theta$ denotes the  K\"ahler angle, see \cite{bernig_fu11}*{Corollary~3.8}. Moreover, the spaces $\Val_{2n-k}^{U(n)}$ are also $2$-dimensional and spanned by two valuations
$\mu_{2n-k,n-k}$ and $\mu_{2n-k,n-k+1}$ satisfying
$$\Klain_{\mu_{2n-k,n-k}}= 1-\cos^2 \theta \qquad \text{and}\qquad \Klain_{\mu_{2n-k,n-k+1}} =\cos^2 \theta.$$

The following lemma expresses the valuations $\varphi_\theta$ and $\psi_\theta$ defined in the introduction in terms of $\mu_{k,0}$ and $\mu_{k,1}$.

\begin{lemma} \label{lem:quermassintegrals}
$$\varphi_\theta = \frac{1}{4n(n-1)} \left(   \left(2n-1-\cos^2\theta\right)  \mu_{2,0} + 2(n-1)\left( 1 + \cos^2 \theta \right) \mu_{2,1} \right)$$
and
$$\psi_\theta= \frac{2^{n-2} (n-3)!}{n\pi (2n-3)!!} \left( \left( 2n-3-\cos^2\theta \right) \mu_{3,0} + 2(n-2)\left(1
+ \frac{1}{3} \cos^2 \theta \right) \mu_{3,1} \right).$$  
\end{lemma}
 
\begin{proof} Fix $E\in\Grass_{k}(\theta)$ and let $B_{E^\perp}$ be the unit ball in $E^\perp$. 
Observe that 
$$\vol_{2n}(K+ rB_{E^\perp}) = \omega_{2n-k}  \vol_k(K|E) r^{2n-k} + O(r^{2n-k-1})$$
and hence
\begin{align*}\int_{\Grass_k(\theta)} \vol_k(K|E)\; dE &= \frac{1}{\omega_{2n-k}}\lim_{r\to \infty} \frac{1}{r^{2n-k}}\int_{U(n)}  \vol_{2n}(K+ g\,  rB_{E^\perp}) \; dg\\
  &= \frac{1}{\omega_{2n-k}}\lim_{r\to \infty} \frac{1}{r^{2n-k}}\int_{U(n)\ltimes \CC^n}  \chi(K \cap   g'\, r B_{E^\perp}) \; dg',
\end{align*}
where $\chi$ is the Euler characteristic.
The integral on the right-hand side can be evaluated using the principal kinematic formula for the unitary group established by Bernig and Fu \cite{bernig_fu11}.  Since an even valuation is uniquely determined by its 
Klain function, it suffices to check the formula  for $\varphi_\theta$  only for $2$-dimensional convex bodies $K$. 
The $(2,2n-2)$ bi-degree part of the principal kinematic formula is given by
\begin{align*}
 \frac{1}{4n(n-1)}  &\bigg[ (2n-1) \mu_{2,0}\otimes \mu_{2n-2,n-2}  + 2(n-1)\mu_{2,0}\otimes \mu_{2n-2,n-1}\\
&\qquad\qquad\qquad + 2(n-1)\mu_{2,1}\otimes \mu_{2n-2,n-2} +      4(n-1) \mu_{2,1}\otimes \mu_{2n-2,n-1}   \bigg],
\end{align*}
see \cite{bernig_fu11}*{p. 941}.
Since $\mu_{2n-2,n-2}(B_{E^\perp}) =\omega_{2n-2} (1-\cos^2 \theta )$ and $ \mu_{2n-2,n-1}(B_{E^\perp}) = \omega_{2n-2} \cos^2\theta$, the claim follows.

The formula for $\psi_\theta$ is  proved in the same way using that  the  $(3,2n-3)$ bi-degree part of the principal kinematic formula is given by
\begin{align*}
 &\frac{2^{n-2} (n-3)!}{n\pi (2n-3)!!}&  \\ &\qquad\qquad\bigg[ (2n-3) \mu_{3,0}\otimes \mu_{2n-2,n-3}  + 2(n-2) \mu_{3,0}\otimes \mu_{2n-3,n-2}\\ 
&\qquad\qquad\qquad+  2(n-2)\mu_{3,1}\otimes \mu_{2n-3,n-3} +    \frac{8(n-2)}{3} \mu_{3,1}\otimes \mu_{2n-3,n-2}   \bigg].
\end{align*}

\end{proof}

\begin{corollary}
The valuation $\mu= c_0 \mu_{2,0} + c_1 \mu_{2,1}$ belongs to the convex cone generated by $\varphi_\theta$ with $\theta$ satisfying \eqref{eq:theta2} if and only if 
$$2(n-1) c_0\leq (2n-1)c_1\qquad \text{and} \qquad (4n+1)c_1 \leq 2(3n+1) c_0.$$
The valuation $\mu= c_0 \mu_{3,0} + c_1 \mu_{3,1}$ belongs to the convex cone generated by $\psi_\theta$ with $\theta$ satisfying \eqref{eq:theta3}
 if and only if 
$$2(n-2)c_0\leq (2n-3)c_{1}\qquad \text{and}\qquad 5c_1\leq 6c_{0}.$$
\end{corollary}

For every $\mu\in \Val_k$ which is given by integration with respect to the normal cycle (see, e.g., \cite{alesker_fu08} for this notion) and every convex body $K$ there exists a signed Borel measure $S_\mu(K)$ 
on the unit sphere of $V$, called the area measure associated to $\mu$, such that
$$\mu(K,\ldots, K,L) = \frac{1}{k}  \dt \mu(K+tL)  = \frac{1}{k}\int h_L\; dS_\mu(K),$$
for every convex body $L$. Explicitly, if $\mu=\int_{N(K)}\omega$, then 
\begin{equation}\label{eq:def_area_meas} S_\mu(K) = \pi_{2*} (N(K) \llcorner ( T\lrcorner D\omega)),\end{equation}
where $\pi_2\colon V\times V\to V$, $\pi_2((u,v))=v$,  $N(K)$ is the normal cycle of $K$, $T$ denotes the Reeb vector field on the sphere bundle of $V$, and $D$ is the Rumin differential, see Proposition~2.2 of \cite{wannerer13}.

Observe that $K\mapsto S_\mu(K)$ is a translation-invariant, $(k-1)$-homogeneous valuation with values in the space of signed Borel measures on the unit sphere, which is continuous: If $K_i\to K$ with respect to the Hausdorff metric,
 then $S_\mu(K_i)\to S_\mu(K)$ with respect to the weak-* topology (see Lemma~2.4 of \cite{wannerer13+}). Hence, by a result of McMullen \cite{mcmullen77}*{Theorem 14}, 
there exists a polarization of $S_\mu$, which, for the sake of simplicity, we denote again by $S_\mu$. More precisely,
there exists a unique  map $S_\mu$ from the $(k-1)$-fold product $\calK(V)\times \cdots \times \calK(V)$ to the space of signed Borel measures on the unit sphere such that
(i) $S_\mu(K,\ldots,K)=S_\mu(K)$; (ii) $S_\mu$ is symmetric in its arguments; (iii) for every $K,L,K_2,\ldots, K_{k-1}$ and  $s,t > 0$
\begin{equation}\label{eq:linearityS}S_\mu( sK+tL,K_2\ldots, K_{k-1}) =s S_\mu( K,K_2\ldots, K_{k-1}) + tS_\mu(L,K_2\ldots, K_{k-1});  \end{equation}
and (iv) for all convex bodies $K_1,\ldots,K_{k-1},L$ 
$$\mu(K_1,\ldots, K_{k-1}, L)= \frac{1}{k} \int h_L \; dS_\mu(K_1,\ldots,K_{k-1}).$$
Moreover, since $\mu(L,K_1,K_2,\ldots, K_{k-1})= \mu(K_1,L,K_2\ldots, K_{k-1})$, property (iv)  implies
\begin{equation}\label{eq:symmetryS}\int h_L \; dS_\mu(K_1,K_2,\ldots, K_{k-1}) = \int h_{K_1} \; dS_\mu(L,K_2,\ldots, K_{k-1}).\end{equation}
We call the polarization of $S_\mu$ the mixed area measure associated to $\mu$.

\section{Elliptic differential operators associated to unitary valuations}

In the following we always assume that $\mu\in \Val_k^{U(n)}$ for $k=2,3$ and $ n\geq k$. 
In a first step, we associate to every valuation $\mu= c_0 \mu_{k,0} + c_1 \mu_{k,1}$
a polynomial function $p_{\mu}$ on $\Sym^2 (\RR\oplus \CC^{n-1})$, the space of symmetric bilinear forms on $\RR\oplus\CC^{n-1}$. To this end we choose an
 orthonormal basis  $\{e_{\overline 1},e_{2},e_{\overline 2},\dots,e_{n},e_{\overline{n}}\}$ of  $\RR\oplus\CC^{n-1}$ such that $e_{\overline 1}$ is an element of the first summand and $Je_{i}=e_{\overline i}$. Here and in the following 
$J$ denotes the standard complex structure on $\CC^n$. 
With respect to this basis a bilinear form $A\in \Sym^2 (\RR\oplus \CC^{n-1})$ is represented by a matrix $(A^i_j)$. We denote by $A^{i_{1}\dots i_{k}}_{j_{1}\dots j_{k}}$ the determinant of the submatrix of $A$ obtained from the
 rows $i_{1},\dots,i_{k}$ and columns $j_{1},\dots,j_{k}$. For $\mu=c_0\mu_{2,0} + c_1\mu_{2,1}$, we define the polynomial $p_{\mu}$ by
$$p_{\mu}(A)= \frac{1}{\omega_{2n-2}} \bigg( ((2n-1)c_1- 2(n-1)c_0) A_{\overline 1} ^{\overline  1} + (2c_0-c_1) \sum_{i=2}^n\left(A_{i}^{i} + A_{\overline i}^{\overline i}\right) \bigg)$$
and, for $\mu=c_0\mu_{3,0} + c_1\mu_{3,1}$, by
\begin{align*}p_{\mu}(A)=\frac{1}{\omega_{2n-3 }}&\bigg( ((2n-3)c_1 -2(n-2)c_0)  \sum_{i=2}^n\left(A_{\overline 1 i}^{\overline1 i}+ A_{\overline 1 \overline i}^{\overline{1}\overline{i}}\right) \\
&\qquad+ (3c_{0}-2c_{1}) \sum_{2\leq i<j\leq n} \left(A_{ij}^{ij}+ A_{i\overline j}^{i\overline j} +A_{\overline i j}^{\overline i j} + A_{\overline i \overline j}^{\overline i \overline j} -2 A_{j\overline j}^{i \overline i}\right)\\
  & \qquad +c_{1} \sum_{2\leq i,j\leq n} A_{j\overline j}^{i \overline i} \bigg).
\end{align*}
Note that the definition of $p_\mu$ does not depend on the particular choice of the orthonormal basis $\{e_{\overline 1},e_{2},e_{\overline 2},\dots,e_{n},e_{\overline{n}}\}$ with the above properties.

For every $u\in S^{2n-1}$, choose an orthonormal basis $\{e_{\overline1}, e_2, e_{\overline 2},\dots,e_{n},e_{\overline n}\}$ of $T_u S^{2n-1}$ such that $Ju=e_{\overline 1}$ and $J e_{i} = e_{\overline{i}}$. 
If $f$ is a $C^2$ function on the unit sphere, we define 
$$D_{\mu}(f)=p_{\mu}(\overline\nabla^2 f+f\overline g),$$
where $\overline g$ denotes the canonical metric on the unit sphere and $\overline\nabla$ the covariant derivative with respect to this metric. 
If $K$ is a convex body with $C^2$ support function, then we also write $D_{\mu}(K)$ instead of  $D_{\mu}(h_{K})$.
In the case $\mu =c_0\mu_{3,0} + c_1\mu_{3,1}$ we consider also the polarization of the $2$-homogeneous polynomial $p_\mu$, again denoted by $p_\mu$, and define
$$D_{\mu}(f_1,f_2)=p_{\mu}(\overline\nabla^2 f_1+f_1\overline g,\overline\nabla^2 f_2+f_2\overline g),$$
for $C^2$ functions $f_1, f_2$ on the unit sphere.  Note that $D_\mu(f,f)= D_\mu(f)$. If $K,L$ are convex bodies with $C^2$ support functions, we write  $D_{\mu}(K,L)$ instead of $D_{\mu}(h_{K},h_{L})$.

\begin{proposition}\label{prop:Smu_eq_Dmu}
If $K$ is a convex body with support function in $C^2(S^{2n-1})$, then 
\begin{equation}\label{eq:Smu_eq_Dmu} dS_\mu(K) = D_\mu(K) \; du,\end{equation}
where $du$ denotes the Riemannian measure on the sphere. 
\end{proposition}

For the proof of \eqref{eq:Smu_eq_Dmu} we have to introduce more notation. Choose an orthonormal basis
$\{e_1,e_{\overline 1},e_{2},e_{\overline 2},\dots,$ $e_{n},e_{\overline{n}}\}$   of  $\CC^{n}$ such that $Je_{i}=e_{\overline i}$ and denote by
 $$(x_1,y_1,\ldots,x_n,y_n,\xi_1,\eta_1,\ldots,\xi_n,\eta_n)$$ the corresponding coordinates on $\CC^n\oplus\CC^n$. 
The $1$-forms
\begin{align*}
	\alpha & = \sum_{i=1}^n \xi_idx_i +\eta_i dy_i,\\
	\beta & = \sum_{i=1}^n \xi_idy_i -\eta_i dx_i,\\
	\gamma & = \sum_{i=1}^n \xi_id\eta_i -\eta_i d\xi_i,
\end{align*}
and the $2$-forms 
\begin{align*}
    \theta_0&=\sum_{i=1}^n d\xi_i\wedge d\eta_i,\\
    \theta_1&= \sum_{i=1}^n dx_i\wedge d\eta_i- dy_i\wedge d\xi_i,\\
    \theta_2&= \sum_{i=1}^n dx_i\wedge dy_i,\\
\end{align*}
are $U(n)$-invariant and hence do not depend on the choice of basis used for their definition. 
The restriction of these forms  to $\CC^n\times S^{2n-1}$ together with the K\"ahler form on $\CC^n$ generate the algebra of translation- and $U(n)$-invariant forms on the  sphere bundle $\CC^n\times S^{2n-1}$, see \cite{bernig_fu11}.  

For non-negative integers $k,q$  with $\max\{0,k-n\}\leq q\leq \frac{k}{2}<n$ Bernig and Fu \cite{bernig_fu11} define the $(2n-1)$-forms
\begin{align*}
\beta_{k,q}&=c_{n,k,q}\beta\wedge\theta_0^{n-k+q}\wedge\theta_1^{k-2q-1}\wedge\theta_2^q,\qquad q<\frac{k}{2},\\
\gamma_{k,q}&=\frac{c_{n,k,q}}{2}\gamma\wedge\theta_0^{n-k+q-1}\wedge\theta_1^{k-2q}\wedge\theta_2^q, \qquad k-n< q,
\end{align*}
where 
$$c_{n,k,q}=\frac{1}{q!(n-k+q)!(k-2q)!\omega_{2n-k}}.$$ 
In terms of integration over the normal cycle, 
$$\mu_{k,q}(K) = \int_{N(K)} \beta_{k,q} =  \int_{N(K)} \gamma_{k,q}.$$

Let $K$ be a convex body with $C^1$ boundary. We denote by $\nu\colon \partial K \to S^{2n-1}$ the Gauss\ map and by $\overline \nu\colon \partial K \to \CC^n\times S^{2n-1}$, $\overline\nu (x) = (x,\nu(x))$,  
the graphing map.
If $K\in C^2_+$, which we assume in the following, then the Gauss map is a $C^1$-diffeomorphism. 
Fix now a point $u\in S^{2n-1}$ and put $x=\nu^{-1}(u)$. By $U(n)$-invariance, we may assume that $u=e_1$. Under this assumption we have at the point $u$,
$$(\overline\nu \circ \nu^{-1})^* dx_i= r^i_{\overline 1} dy_1 +   \sum_{j=2}^n( r^i_{j} dx_j + r^i_{\overline j} dy_j),\qquad 1<i\leq n,$$
$$(\overline\nu \circ \nu^{-1})^*dy_i=  r^{\overline i}_{\overline1} dy_1 + \sum_{j=2}^n(r^{\overline i}_{ j} dx_j + r^{\overline i}_{ \overline j} dy_j),\qquad 1\leq i\leq n,$$
where $(r^i_j)$ is the matrix representing the bilinear form 
$$\left\langle d_{u}\nu^{-1}( X), Y\right\rangle = \overline \nabla ^2 h_K(X,Y) + h_K \left\langle X,Y\right\rangle$$
with $X\in T_{u} S^{2n-1}$ and $Y\in T_{x}\partial K\cong T_{u} S^{2n-1}$. Moreover,
$$(\overline\nu \circ \nu^{-1})^*\alpha=0, \quad (\overline\nu \circ \nu^{-1})^*\beta=(\overline\nu \circ \nu^{-1})^*dy_1, \quad (\overline\nu \circ \nu^{-1})^*\gamma=(\overline\nu \circ \nu^{-1})^*d\eta_1=dy_1,$$
and 
$$(\overline\nu \circ \nu^{-1})^*d\xi_1=0, \qquad (\overline\nu \circ \nu^{-1})^*d\xi_i =  dx_i, \quad (\overline\nu \circ \nu^{-1})^*d\eta_i= dy_i$$
for $1<i\leq n$.

\begin{lemma}\label{lem:pullbacks} Suppose $K\in C^2_+$. Then
 \begin{align*}
(\overline\nu \circ \nu^{-1})^* \beta_{1,0}& = \frac{1}{\omega_{2n-1} } r_{\overline 1}^{\overline 1} \; du,   \\
(\overline\nu \circ \nu^{-1})^* \gamma_{1,0}& = \frac{1}{2(n-1)\omega_{2n-1} }  \sum_{i=2}^n\left(r_{i}^{i}+ r_{ \overline i}^{\overline i}\right) \; du,\\
(\overline\nu \circ \nu^{-1})^* \beta_{2,0}& =  \frac{1}{2\omega_{2n-2}}  \sum_{i=2}^n\left(r_{\overline 1 i}^{\overline1 i}+ r_{\overline 1 \overline i}^{\overline1\overline i}\right) \; du,  \\
(\overline\nu \circ \nu^{-1})^* \gamma_{2,0}& = \frac{1}{2(n-2)\omega_{2n-2}} \sum_{2\leq i<j\leq n} \left(r_{ij}^{ij}+ r_{i\overline j}^{i\overline j} +r_{\overline i j}^{\overline i j} + r_{\overline i \overline j}^{\overline i \overline j} -2 r_{j\overline j}^{i \overline i}\right) \; du,\\
(\overline\nu \circ \nu^{-1})^* \gamma_{2,1}& = \frac{1}{2(n-1) \omega_{2n-2}}   \sum_{2\leq i,j\leq n} r_{j\overline j}^{i \overline i} \; du,
\end{align*}
where $du$ denotes the Riemannian volume form.
\end{lemma}
\begin{proof} Using the above relations, the proof is a straightforward computation.
\end{proof}

\begin{proof}[Proof of Proposition~\ref{prop:Smu_eq_Dmu}]  If $\mu=\int_{N(\; \cdot \;)} \omega$, then, by equation \eqref{eq:def_area_meas},
$$\int_{S^{2n-1}} f \; dS_\mu(K)= \int_{N(K)} \pi_2^*f \, \omega',$$
where $\omega'= T\lrcorner D\omega$ and $D\omega$ denotes the Rumin differential of $\omega$.  
Bernig and Fu have  computed  
$T\lrcorner D\omega$ for each of the invariant forms $\beta_{k,q}$ and $\gamma_{k,q}$, see Propositions~3.4 and 4.6 of \cite{bernig_fu11}.
Using this, we obtain
 $$T\lrcorner D \omega = \frac{\omega_{2n-1}}{\omega_{2n-2}} \left( ((2n-1)c_1-2(n-1)c_0) \beta_{1,0} + 2(n-1)(2c_0-c_1)  \gamma_{1,0} \right)$$
for  $\mu=c_0\mu_{2,0} + c_1\mu_{2,1}$ and
\begin{align}T \lrcorner D\omega = \frac{2\omega_{2n-2}}{\omega_{2n-3}} \Big( &((2n-3)c_1 -2(n-2)c_0) \beta_{2,0}+ (n-2)(3c_0-2c_1) \gamma_{2,0} \label{eq:RuminD} \\
  & \qquad \qquad + (n-1)c_1 \gamma_{2,1} \Big)\notag
\end{align}
for $\mu=c_0\mu_{3,0} + c_1\mu_{3,1}$.
Hence, if $K\in C^2_+$, then \eqref{eq:Smu_eq_Dmu} follows from
$$\int_{N(K)} \pi_2^*f \, \omega'= \int_{S^{2n-1}} f\, (\overline \nu \circ \nu^{-1})^* \omega',$$
and Lemma~\ref{lem:pullbacks}.

If $K$ is a convex body whose support function is merely $C^2$,
then for every $\varepsilon>0$ the Minkowski sum $K_\varepsilon= K + \varepsilon B$ is in $C^2_+$.  Therefore, $S_\mu(K_\varepsilon)=D_\mu(K_\varepsilon)\; du$. Since  
$S_\mu(K_\varepsilon)$ and $D_\mu(K_\varepsilon)$ are polynomial in $\varepsilon>0$ by \eqref{eq:linearityS} and $h_{K_\varepsilon} = h_K + \varepsilon$, letting $\varepsilon\to 0$ concludes the proof.
\end{proof}

For later use we note that \eqref{eq:symmetryS} and \eqref{eq:Smu_eq_Dmu} imply
\begin{equation}\label{eq:Dmu_symmetry} \int f_1 D_\mu(f_2,f_3) \; du = \int f_2 D_\mu(f_1,f_3) \; du \end{equation}
for all $C^2$ functions $f_1$, $f_2$, and $f_3$.

\smallskip

A homogeneous polynomial $P$ of degree $m$ defined on $\RR^n$ is called hyperbolic in direction $a\in\RR^n$ if $P(a)>0$ and for every $x \in\RR^n$ the univariate polynomial 
$$t\mapsto P(ta + x)$$
has exactly $m$ real roots (counted with multiplicities). If $P$ is hyperbolic in direction $a$, then $\Gamma=\Gamma(P,a)$ denotes the connected component of the set $\{P>0\}$ containing $a$ and is called 
the hyperbolicity cone of $P$. It was shown by G\aa rding \cite{garding59} that $\Gamma$ is a convex cone and that $P$ is hyperbolic in direction $b$ for every $b\in \Gamma$. 

For example, $x\mapsto x_1\cdots x_n$ is a homogeneous 
polynomial on $\RR^n$ which is hyperbolic in direction $(1,\ldots, 1)$. Since every symmetric $n\times n$ matrix $A$ has $n$ real eigenvalues, the determinant $A\mapsto \det A$ is hyperbolic in 
direction of the identity matrix.

\begin{proposition} \label{prop:hyperbolicity}
Suppose $\mu=c_0\mu_{3,0} + c_1\mu_{3,1}$. Then
\begin{equation}\label{eq:c_range}2(n-2) c_0< (2n-3)c_{1}\qquad \text{and}\qquad 5c_1<6c_{0},\end{equation}                                                                           
if and only if for every $A, X\in \Sym^2(\RR \oplus \CC^{n-1})$ with $A$ positive definite
\begin{equation} \label{eq:hyperbolicity}  p_\mu(A,X)=0 \quad \Rightarrow\quad p_\mu(X,X)\leq 0\end{equation}
 and equality holds if and only if $X=0$.
\end{proposition}

\begin{proof}
We show first that \eqref{eq:c_range} implies 
\begin{equation}\label{eq:discriminant}p_\mu(I,X)=0 \quad \Rightarrow\quad p_\mu(X,X)\leq 0\end{equation}
for every $X\in \Sym^2(\RR \oplus \CC^{n-1})$, where  $I$ denotes the bilinear form corresponding to the identity matrix $(\delta^i_j)$.
Note that since $p_\mu(I)>0$ and
$$p_\mu(t I+X)= p_\mu(X,X) + 2 tp_\mu(I,X) + t^2 p_\mu(I,I),$$
the claim \eqref{eq:discriminant} is equivalent to the statement that $p_\mu$ is hyperbolic in direction $I$. Indeed, if $p_\mu$ is hyperbolic in direction $I$, then \eqref{eq:discriminant} holds. Conversely, given any 
$X\in \Sym^2(\RR \oplus \CC^{n-1})$, put $X'= X-\lambda I$ with $\lambda = p_\mu(I,X)/p_\mu(I,I)$. Thus, $p_\mu(I,X')=0$ and hence, by \eqref{eq:discriminant}, 
\begin{equation}\label{eq:garding_ineq_I} p_\mu(X',X')= p_\mu(X,X)- \frac{p_\mu(I,X)^2}{ p_\mu(I,I)}\leq 0.\end{equation}                                                                                                         
This shows that $p_\mu$ is hyperbolic in direction $I$.

By the $\{1\}\times U(n-1)$-invariance of \eqref{eq:discriminant},
we may assume that  $X=(X^i_j)$ satisfies
$$\left\{\begin{array}{l}
          X^{ 2}_{\overline 2} =0,\\
	  X^{2}_{\overline 3}=X^{\overline 2}_{\overline 3} = X^{3 }_{\overline 3}=0, \\
    \phantom{ X^{2}_{\overline 3}a} \vdots \\
	  X^{2}_{\overline n}=X^{\overline 2}_{\overline n} =\cdots = X^{n-1}_{\overline{ n}}= X^{\overline{n-1} }_{\overline{ n}}=X^{n }_{\overline n}=0. \\
         \end{array}\right.$$
In this case, the minors $X_{j\overline j}^{i\overline i}$ vanish for $i\neq j$  and hence
$$p_\mu(X)\leq p_\mu(\tilde X),$$
where $(\tilde X^i_j)$ has the same diagonal entries as $( X^i_j)$, but all off-diagonal entries are $0$.

The condition $p_\mu(I,X)=0 $ means explicitly that 
\begin{align*} p_\mu(I,X)=  \frac{1}{\omega_{2n-3} } &\bigg[ (n-1)((2n-3)c_1-2(n-2) c_0)  X^{\overline 1}_{\overline 1} \\ 
 &  \qquad \qquad + (2(n-2) c_0-(n-3)c_1) \sum_{i=2}^n\left(X_{i}^{i} + X_{\overline i}^{\overline i}\right)\bigg]=0
\end{align*}
and, hence, 
$$X_{\overline 1}^{\overline 1} = -\frac{2(n-2) c_0-(n-3)c_1}{(n-1)((2n-3)c_1-2(n-2) c_0)} \sum_{i=2}^n\left(X_{i}^{i} + X_{\overline i}^{\overline i}\right).$$
Thus $p_\mu(\tilde X)$ is in fact a homogeneous polynomial of degree $2$ in the variables $ X_2^2, X_{\overline 2}^{\overline 2},\ldots, X_n^n, X_{\overline n}^{\overline n}$, 
\begin{align*}\omega_{2n-3} p_\mu(\tilde X)
    & =\frac{a}{2} \sum_{i=2}^n  \left( (X^i_i)^2+ (X^{\overline i}_{\overline i})^2 \right) + b \sum_{i=2}^n  X^i_i X^{\overline i}_{\overline i} \\
& \qquad\qquad \qquad + c \sum_{2\leq i< j\leq n} \left( X^i_i X^{ j}_{ j} + X^i_i X^{\overline j}_{\overline j} +  X^{\overline i}_{\overline i} X^j_j + X^{\overline i}_{\overline i} X^{\overline j}_{\overline j}  \right)\\
& =q(X_2^2, X_{\overline 2}^{\overline 2},\ldots, X_n^n, X_{\overline n}^{\overline n})
\end{align*}
with
$$ a = -\frac{2(2(n-2) c_0-(n-3)c_1)}{(n-1)}, \quad b  = c_1+a  , \quad c = (3c_0-2c_1)+ a.$$

In order to show $q\leq 0$, it will be sufficient to compute the eigenvalues of the Hessian of $q$. Since 
$$\Hess q =\left(\begin{array}{ccccccccc} a & b & c  & c & &  \multicolumn{4}{c}{\multirow{4}{*}{c}}   \\
		  b & a & c & c &  &  & & &\\
		  c & c	& a & b &  &   &  & &\\
		  c & c & b & a &  &   & & & \\
		   & &  &   & \ddots & & & &  \\
		  \multicolumn{4}{c}{\multirow{4}{*}{c}}      &	&  a & b & c & c \\
							 & & & & &  b & a & c & c \\
							 & & & & &  c &  c &  a & b \\
							 & & & & &  c &  c &  b & a\\
		  
  \end{array}\right),$$
we conclude that $\Hess q$ has the eigenvalues 
\begin{align*}
 a-b & = -c_1,\\
 a+ b-2c& = 5c_1-6c_0,\\
 a+ b +2(n-1)c & =-\frac{2(n+1)(n-3)}{n-1}c_0-\frac{3n+5}{n-1}c_1,
\end{align*}
with multiplicities $n-1$, $n-2$, and $1$. By assumption \eqref{eq:c_range}, all eigenvalues are negative and hence $q\leq 0$.

Next, we claim that $p_\mu(A)>0$ if $A$ is positive definite. Again by  $\{1\}\times U(n-1)$-invariance, we may assume that  $A_{j\overline j}^{i\overline i}=0$ for $i\neq j$. Since 
$(2n-3)c_1-2(n-2) c_0>0$ and $3c_{1}-2c_{0}>0$, we conclude that
$p_\mu(A)>0$. Thus every positive definite bilinear form $A$ is contained in the hyperbolicity cone $\Gamma(p_\mu,I)$ and hence $p_\mu$ is hyperbolic in direction $A$. This implies \eqref{eq:hyperbolicity}.

Consider now the problem of maximizing $p_\mu(X)$ subject to the condition $g(X):=p_\mu(A,X)=0$. By the method of Lagrange multipliers, if $X$ maximizes $p_\mu$, then there exists some number $\lambda$ such that
$$\nabla p_\mu(X) = \lambda \nabla g(X)\quad \text{and} \quad g(X)=0.$$
A straightforward computation shows that $\nabla p_\mu(X) = \lambda \nabla g(X)$ is equivalent to
$$2 X + \lambda A=0.$$
Since $p_\mu(A,A)>0$, $g(X)=0$ implies $\lambda=0$ and hence $X=0$. 

Conversely, to see that \eqref{eq:hyperbolicity} implies \eqref{eq:c_range}, choose $A=I$ and plug  $X$ diagonal or of rank at most $2$ into \eqref{eq:hyperbolicity}.
\end{proof}

For later use we remark that \eqref{eq:hyperbolicity} is equivalent to the statement that for every $A, X\in \Sym^2(\RR \oplus \CC^{n-1})$ with $A$ positive definite
\begin{equation}\label{eq:garding_ineq}p_\mu(A,X)^2\geq p_\mu(A)p_\mu(X)\end{equation}                                                                        
and equality holds if and only if there exists $\lambda\in \RR$ such that $X=\lambda A$. Indeed, the proof of the equivalence of \eqref{eq:discriminant} and \eqref{eq:garding_ineq_I} with $I$ replaced by $A$
yields the equivalence of  \eqref{eq:hyperbolicity} and \eqref{eq:garding_ineq}.

Let $M$ be a smooth manifold. A linear map $D\colon C^2(M) \to C(M)$ is called linear differential operator of order at most $2$ if for every coordinate neighborhood $U$ in $M$
with local coordinates $(x^1,\ldots, x^n)$ there exist continuous functions  $a^{ij}=a^{ji}$, $b^i$, $c$ such that given any $f\in C^2(M)$ 
 the restriction $D f|_U$ to $U$  is given by
\begin{equation}\label{eq:2nd_order_operator}D f|_U= \sum_{i,j=1}^n a^{ij} \frac{\partial^2 f}{\partial x^i \partial x^j} + \sum_{i=1}^n b^i \frac{\partial f}{\partial x^i} + c f. \end{equation}
The operator $D$ is called elliptic if 
$$ a^{ij}\xi_{i} \xi_{j} \neq 0$$
for every $\xi\in \RR^n$ with $\xi\neq 0$, see [Aubin, p. 125]. 
The principal symbol of $D$ is the contravariant, symmetric tensor $\sigma(D)^{ij} =a^{ij}$.

\begin{corollary}\label{cor:elliptic_operator} Suppose $\mu=c_0\mu_{3,0} + c_1\mu_{3,1}$,
$$2(n-2)c_0< (2n-3)c_{1}\qquad \text{and}\qquad 5c_1<6c_{0},$$
and $M\in C^2_+$.
Then the operator $f\mapsto D_{\mu,M}f:= D_\mu(M,f)$ is a
formally self-adjoint,  elliptic linear  differential operator of order at most $2$. Moreover,                                                                                                                                                 
$$D_\mu(M,f)=0\quad \Rightarrow\quad D_\mu(f,f)\leq 0$$
and equality holds if and only if $f$ is the restriction of a linear function to the unit sphere.
\end{corollary}
\begin{proof}
The symmetry of $\mu(K,L,M)$ implies that the operator $D_{\mu,M}$ is formally self-adjoint. Indeed, every $C^2$ function  can be expressed as the difference of two $C^2$ support functions, and hence, for
 $f=h_{K_1} - h_{K_2}$ and $g=h_{L_1}-h_{L_2}$ 
\begin{align*}(f,D_\mu(M, g) )_{L^2}&= \mu(M,K_1,L_1)-\mu(M,K_1,L_2)-\mu(M,K_2,L_1)+ \mu(M,K_2,L_2)\\
  &=  (D_\mu(M ,f),g)_{L^2}.
\end{align*}
In order to prove ellipticity, fix a point $p\in S^{2n-1}$ and choose normal coordinates $x^1,\ldots, x^{2n-1}$ for $p$ such that 
$\frac{\partial }{\partial x^1}, \ldots, \frac{\partial }{\partial x^{2n-1}}$ is a basis of the form $\{e_{\overline 1},e_{2},e_{\overline 2},\ldots, e_{n},e_{\overline n}\}$ for $T_pS^{2n-1}$. At the point $p$ we have
$$\sigma(D_{\mu,M})^{ij}\xi_{i} \xi_{j}= p_{\mu}(\overline\nabla^2 h_{M}+h_{M}\overline g, \xi^*\xi)$$
and the last expression is, by Proposition~\ref{prop:hyperbolicity}, zero if and only if $\xi=0$. Hence $D_{\mu,M}$ is elliptic.

Since every linear functional is the support function of some point $P$, \eqref{eq:val_kernel} yields $D_\mu(M,f)= D_\mu(f,f)=0$. Conversely, if $D_\mu(M,f)= D_\mu(f,f)=0$ then $\overline \nabla^2 f+f \overline g=0$
  by Proposition~\ref{prop:hyperbolicity}.
In particular, $\tr_{\overline g}(\overline \nabla^2 f+f \overline g)=\Delta f+(2n-1)f=0$,
 that is, $f$ is an eigenfunction of the Laplace-Beltrami operator on the sphere with eigenvalue $-2n+1$. As is well known, see \eqref{eq:specLaplacian},
 this is possible if and only if $f$ is the restriction of a linear functional to the unit sphere.
\end{proof}

In the following $JN$ will denote the canonical vector field on $S^{2n-1}\subset \CC^n$ given by $JN(u)= Ju$. Since the trajectories of the vector field $JN$ are geodesics,
  $\overline{\nabla}_{JN}{JN}=0$ and, hence, $\overline{\nabla}^2f (JN,JN)= JN(JN f)$.
Consequently, we have for $\mu=c_0\mu_{2,0} + c_1\mu_{2,1}$,
\begin{equation} \label{eq:Dmu2}D_{\mu}=\frac{1}{\omega_{2n-2}} \Big[ 2n(c_1-c_0)JN(JN )+(2c_0-c_1)\Delta +(2(n-1)c_0+c_{1})\Big]. \end{equation}
Similarly, if $M=B$ is the unit ball in $\CC^n$, we have
\begin{equation}\label{eq:DmuB}D_{\mu,B}=\frac{1}{\omega_{2n-3}} \Big[ (a-b)JN(JN )+b\Delta +a+2(n-1)b\Big]
\end{equation}
with
$$a= (n-1)((2n-3)c_1-2(n-2) c_0) \qquad\text{and}\qquad  b=2(n-2) c_0-(n-3)c_1.$$

We denote by $\calH_m=\calH_m(S^{2n-1})$, $n\geq 2$, the space of spherical harmonics of degree $m$, i.e.~the space of restrictions of harmonic, $m$-homogeneous polynomials 
$P\in \CC[x_1,y_1,\ldots, x_n,y_n]$
 to the unit sphere. It is well known that
\begin{equation}\label{eq:specLaplacian}\Delta f=  -m(m+2n-2) f \qquad \text{for} \qquad f\in \calH_m(S^{2n-1}).\end{equation}                                                                                                           
For non-negative integers $k,l$  we denote by $\calH_{k,l}$ the space of harmonic polynomials $P\in \CC[x_1,y_1,\ldots, x_n,y_n]=\CC[z_1,\overline z_1,\ldots, z_n,\overline z_n]$ restricted to the unit sphere for which
$$P(\lambda z)= \lambda^k \overline\lambda^l P(z) \quad\text{for}\quad \lambda\in \CC.$$
Clearly, $\calH_{k,l}\subset \calH_{k+l}$. The space $\calH_{k,l}$ is called the space of spherical harmonics of bi-degree $(k,l)$. Under the canonical action of the unitary group $U(n)$ on $L^2(S^{2n-1})$, the spaces
$\calH_{k,l}$ are invariant and irreducible. In particular, we have the decompositions 
$$\calH_m=\bigoplus_{k+l=m} \calH_{k,l} \qquad \text{and}\qquad  L^2(S^{2n-1})= \bigoplus_{k,l} \calH_{k,l}$$
into pairwise orthogonal, irreducible subspaces.

Fix some point $e\in S^{2n-1}$. A function $P$ is called a spherical function with respect to $U(n-1)$ if $P$ is contained in some $\calH_{k,l}$, $P$ is $U(n-1)$-invariant, and $P(e)=1$.
The existence of a unique spherical function in every $\calH_{k,l}$ follows from Frobenius reciprocity and the fact that irreducible $U(n)$-representations decompose with multiplicity $1$ under $U(n-1)$, see
\cite{knapp02}*{p. 569}.
 One can show that the unique spherical 
function in $\calH_{k,l}$, denoted by $P_{k,l}( (w,e) )$, is given by
$$\begin{array}{ll}
 P_{k,l}( re^{i\theta}) =  (re^{i\theta})^{k-l}  Q_l(k-l, n-2, r^2)  & \qquad \text{if}\  k\geq l\text{; and}\\
 P_{k,l}  =  \overline{P_{l,k}}  & \qquad \text{if}\  l>k.\\
\end{array}$$
Here $\{Q_l(a,b,t)\colon l=0,1,2,\ldots\}$ is the complete set of polynomials in $t$ ($Q_l$ has degree $l$) orthogonal on $[0,1]$ with weight $t^a(1-t)^b\; dt$ and satisfying $Q_l(a,b,1)=1$, $a>-1$, $b>-1$.

The above description of spherical functions is essentially due to Johnson and Wallach \cite{johnson_wallach77}*{Theorem 3.1 (3)}; see also \cite{quinto87} and the references therein for more information 
on these spherical functions.

\begin{lemma}\label{eigenv:JN} For $f\in \calH_{k, l}(S^{2n-1})$,
 $$JN(JN f) = -(k-l)^2f.$$
\end{lemma}
\begin{proof} Fix $e=e_1$. Since $f\mapsto JN f$ is a $U(n)$-intertwining operator, it will be sufficient to compute $JN f$ for $f=P_{k,l}((\;\cdot\;, e))$. Let $a,b\in \CC$ be such that $|a|^2+|b|^2=1$ and $0<|a|<1$, and 
choose $z\in S^{2n-1}$ such that $e \perp z$. Put $w= a e + b z$ and let $\gamma \colon \RR \to S^{2n-1}$ be the curve $\gamma(t) = \cos(t) w + \sin(t) Jw $. Then $(\gamma(t),e)= (\cos t + i\sin t)a=a e^{it}$, $\gamma(0)=w$, 
$\gamma'(0)= Jw= JN_w$, and 
$$ JN f (w)= \dt P_{k,l}((\gamma(t),e)) = i (k-l) f(w).$$
\end{proof}

For the proof of Theorems \ref{th1} and \ref{th2}, we need the following  description of the spectrum of the differential operators  \eqref{eq:Dmu2} and \eqref{eq:DmuB}.

\begin{proposition}\label{prop:tildeD}
Let $D\colon C^2(S^{2n-1})\to C(S^{2n-1})$ be the differential operator defined in  \eqref{eq:Dmu2} or \eqref{eq:DmuB} and denote by $D^\CC$ its extension to $\CC$-valued functions. If 
$$2(n-1) c_0\leq (2n-1)c_1\qquad \text{and} \qquad (4n+1)c_1 < 2(3n+1) c_0$$
or 
$$2(n-2)c_0 \leq (2n-3)c_1\qquad \text{and}\qquad (4n^2-9n -3) c_1< 2(3n^2-5n-2) c_0,$$
respectively, then $D^\CC$ has precisely one positive eigenvalue, which corresponds to the $1$-dimensional space of constant functions, and 
the kernel of $D^\CC$ consists of the restriction of linear functionals to the unit sphere.

Moreover, if $0<  (4n+1)c_1 = 2(3n+1) c_0$, then  the kernel of $D^\CC$ is $\calH_{1,0} \oplus \calH_{0,1} \oplus \calH_{1,1}$. 
\end{proposition}

\begin{proof}
To prove the statement for the operator defined in \eqref{eq:Dmu2}, it suffices, by \eqref{eq:specLaplacian} and Lemma~\ref{eigenv:JN}, to show that in the specified range for $c_0$ and $c_1$,
\begin{equation}\label{eq:eigenDmu2}   -2n(c_1-c_0) (k-l)^2 - (2c_0 -c_1 ) (k+l)(k+l+2n-2) +(2(n-1)c_0+c_1) \end{equation}
is negative if $k+l>1$. To this end put $\alpha = 2n(c_1-c_0)$, $\beta= 2c_0-c_1$, $\gamma=2(n-1)c_0+c_1$, $k+l=m$, $j=|k-l|$, and observe that $\beta>0$ and
$$\frac{\alpha}{\beta} + 2n-1 = \frac{\gamma}{\beta}.$$
Thus \eqref{eq:eigenDmu2} becomes 
$$ \frac{\alpha}{\beta} (1-j^2) -((m+n-1)^2-n^2),$$
which is negative for $1<m$ and $0\leq j\leq m$ if and only if $-\beta\leq \alpha <( 2n+1)\beta$.

Finally, defining $\alpha=a-b$, $\beta=b$, $\gamma=a+2(n-1)b$, and  using that
$\frac{\alpha}{\beta} + 2n-1 = \frac{\gamma}{\beta}$,  we conclude as before that $-\beta\leq \alpha <( 2n+1)\beta$. 
\end{proof}

\section{Proof of the inequalities}

In this section we prove that if $\mu$ belongs to the convex cone generated by the valuations $\psi_\theta$ 
with $\theta$ satisfying \eqref{eq:theta3},
then
\begin{equation}\label{eq:ineq3}\mu(f,L,M)^2 \geq \mu(f,f,M) \mu(L,L,M)\end{equation}
for all convex bodies $L,M$  and all differences of support functions $f$. Moreover, we show that 
\begin{equation}\label{eq:ineq2}\mu(f,L)^2 \geq \mu(f,f) \mu(L,L)\end{equation}
whenever $\mu$ belongs to the convex cone generated by the valuations $\varphi_\theta$ with $\theta$ satisfying \eqref{eq:theta2}.

Since every convex body can be approximated in the Hausdorff metric by convex bodies with non-empty interior, $C^\infty$ boundary, and $C^\infty$ support function, it will suffice to prove \eqref{eq:ineq3} and \eqref{eq:ineq2} for such 
convex bodies and smooth functions $f$.

\begin{proposition} \label{prop:equiv0} Let $L,M$ be convex bodies with non-empty interior, $C^\infty$ boundary, and $C^\infty$ support function, and $f$ a $C^\infty$ function on the unit sphere. 
Suppose $\mu$ belongs to the convex cone generated by the valuations $\psi_\theta$ 
with $\theta$ satisfying \eqref{eq:theta3_open}.                                                                    
Then the condition
$$\mu(f,L,M)=0$$
implies 
$$ \mu(f,f,M)\leq 0$$
and equality holds if and only if $f$ is the restriction of a linear functional to the unit sphere.
 
If $\mu$ belongs to the convex cone generated by the valuations $\varphi_\theta$ with
\begin{equation}\label{eq:theta2_open}0< \cos^2\theta< \frac{n+1}{2n},\end{equation}                                                                       
then the condition
$$\mu(f,L)=0$$
implies 
$$ \mu(f,f)\leq 0$$
and equality holds if and only if $f$ is the restriction of a linear functional to the unit sphere. 
\end{proposition}

The inequalities  \eqref{eq:ineq3} and \eqref{eq:ineq2} are readily implied by Proposition~\ref{prop:equiv0}. Indeed, if $L,M$ are convex bodies with non-empty interior, $C^\infty$ boundary, and $C^\infty$ support function, 
then, by inequality \eqref{eq:garding_ineq},
$$\mu(L,L,M)>0$$ and hence there exists a real number $\lambda$ such that
$$\mu(f, L,M)-\lambda \mu(L,L,M)=0.$$
Put $f'=f -\lambda h_L$. Then $\mu(f',L,M)=0$ by linearity, and hence
$$0\geq \mu(f',f',M)= \mu(f,f,M)- \frac{\mu(f,L,M)^2}{\mu(L,L,M)}.$$

We follow Hilbert \cite{hilbert12}*{Chapter 19} and Aleksandrov \cite{aleksandrov38} to prove Proposition~\ref{prop:equiv0}. By translation-invariance, we may assume that $L$ and $M$ contain the origin
in their interior and hence $h_L, h_M>0$. Moreover, $D_{\mu,M} (L)>0$ by \eqref{eq:garding_ineq}. Consider the eigenvalue problem
\begin{equation}\label{eq:elliptic_eigenvalue}D_{\mu, M}( f)+\lambda\frac{D_{\mu,M} (L)}{h_{L}}f=0.\end{equation}
Since  $\mu$ belongs to the convex cone generated by $\psi_\theta$ with $\theta$ satisfying \eqref{eq:theta3_open},  $D_{\mu,M}$ is, by Corollary~\ref{cor:elliptic_operator}, a formally self-adjoint, elliptic
linear differential operator. Replacing $D_{\mu,M}$ by the formally self-adjoint, elliptic
linear differential operator 
$$\tilde D_{\mu,M} (f)= \left(\frac{h_{L}}{D_{\mu,M} (L)}\right)^\frac{1}{2} D_{\mu, M}\left( \left(\frac{h_{L}}{D_{\mu,M} (L)}\right)^\frac{1}{2}f\right), $$
 the general theory of such operators implies (see, e.g., \cite{aubin98}*{p. 125}) that
there exists an orthonormal basis $\{f_k\}_{k=1}^\infty$ of $L^2\left(S^{2n-1}, \frac{D_{\mu,M}(L)}{h_{L}}\; du\right)$ such that $f_k$ is $C^\infty$ and a solution of \eqref{eq:elliptic_eigenvalue}. 
The set of eigenvalues of \eqref{eq:elliptic_eigenvalue} is countable and discrete and the corresponding eigenspaces are finite-dimensional. Moreover, there are only finitely many negative eigenvalues.

We investigate the set of eigenvalues of \eqref{eq:elliptic_eigenvalue} more closely. 

\begin{proposition}\label{prop:neg_eigenvalues} 
If $\mu$ belongs to the convex cone generated by $\psi_\theta$ with $\theta$ satisfying \eqref{eq:theta3_open}, then $0$ and $-1$ are eigenvalues of \eqref{eq:elliptic_eigenvalue} 
and the corresponding eigenspaces  are spanned by the restriction of linear functionals to the unit 
sphere and by $h_{L}$, respectively. All other eigenvalues of \eqref{eq:elliptic_eigenvalue} are positive.
\end{proposition}
\begin{proof}
Suppose  $\lambda=0$ and that $f$ is a solution of \eqref{eq:elliptic_eigenvalue}. Then, 
$D_{\mu}(M,f)=0$, and Corollary~\ref{cor:elliptic_operator} yields $D_\mu(f,f)\leq 0$. From \eqref{eq:Dmu_symmetry} we have 
$$0=\int  fD_\mu(M,f)=\int  h_{M}D_\mu(f,f)\leq 0,$$
which implies $D_\mu(f,f)=0$.  By Corollary~\ref{cor:elliptic_operator}, this is possible if and only if $f$ is the restriction of a linear functional to the unit sphere.
 
If $\lambda =-1$, then it is clear that $f=h_{L}$ is a solution of \eqref{eq:elliptic_eigenvalue}. We show now that every other solution of \eqref{eq:elliptic_eigenvalue} with $\lambda =-1$ must be a multiple of $h_L$
and that there are no other negative eigenvalues.

We prove this statement first for $L=M=B$, where $B$ denotes the unit ball in $\CC^n$. In this case the eigenvalue problem in \eqref{eq:elliptic_eigenvalue} reduces to 
\begin{equation}\label{eq:D3Ball} D_{\mu,B}(f)+\lambda \frac{n-1}{\omega_{2n-3}}(2(n-2)c_0+3c_1)f=0,\end{equation}
where  $D_{\mu,B}$ is given explicitly by equation \eqref{eq:DmuB} and the constants $c_0,c_1$ arise from $\mu=c_0\mu_{3,0}+ c_1\mu_{3,1}$. 
Since $2(n-2)c_0+3c_1>0$ and $D_{\mu,B}$ and its complexification have the same spectrum, the desired statement follows directly from  Proposition~\ref{prop:tildeD}.

Let $L,M $ now be general convex bodies  with $C^\infty$ boundary and $C^\infty$ support function containing the origin in the interior. 
Since $L, M$ have all principal curvatures strictly positive, see, e.g., \cite{schneider14}*{p. 115},  
$$L_t = (1-t) B + t L \qquad \text{and }\qquad M_t = (1-t) B + t M, \qquad t\in [0,1],$$
are convex bodies  with $C^\infty$ boundary and $C^\infty$ support function containing the origin in the interior.  
Hence $\{\tilde D_{\mu, M_t}\colon t\in[0,1]\}$ is a family of uniformly elliptic, self-adjoint, linear differential operators, i.e.\
$$\sigma(\tilde D_{\mu, M_t})^{ij}\xi_i \xi_j \geq  c \,\overline{g}^{ij}\xi_i \xi_j   \qquad \text{for } \xi\in \RR^{2n-1}$$
with some constant $c>0$ independent of $t$.
We denote by $$ \lambda_1(t)\leq \lambda_2(t)\leq \lambda_3(t)\leq  \cdots$$
the eigenvalues of 
\begin{equation}\label{eq:D_t}\tilde D_{\mu,M_t}(f)+\lambda f=0,\end{equation}
 ordered and repeated according to their multiplicity. Since the family $\{\tilde D_{\mu, M_t}\colon t\in[0,1]\}$ is uniformly elliptic, Theorem~2.3.3 of \cite{henrot06} (which is stated only for bounded domains of $\RR^n$, but the 
proof works also for compact manifolds) guarantees the continuous dependence of $\lambda_k(t)$ on $t$.

Suppose there exists some $t\in[0,1]$ such that $\lambda_2(t)<0$. Put
$$t_0= \inf\{ t\in [0,1]\colon  \lambda_2(t)<0 \}.$$
By continuity, $\lambda_2(t_0)=0$. Moreover, since for every $t\in [0,1]$ the eigenvalue $0$ has multiplicity $2n$, $\lambda_{2n+2}(t_0)>0$. If $t_0<1$, then for $t>t_0$ sufficiently close to $t_0$
we have $\lambda_{2n+2}(t)>0$   and hence  $\lambda_2(t)=0$. This contradicts the definition of $t_0$. We conclude that $\lambda_2(t)=0$ for $t\in [0,1]$. 
\end{proof}

To conclude the proof of Proposition \ref{prop:equiv0} suppose that 
$$\mu(f,L,M)=0.$$ 
Let $f= \sum_{k=1}^\infty f_k$ be the expansion of $f$ into eigenfunctions of \eqref{eq:elliptic_eigenvalue}. 
Here we stipulate that every $f_k$ corresponds to a different eigenvalue $\lambda_k$, ordered by their size. In particular, we have
$\lambda_{1}=-1$ and $\lambda_{2}=0$.  Since  $f_k$ and $f_l$ for $k\neq l$ are orthogonal with respect to the $L^2$ inner product with weight $D_{\mu,M}(L)/h_L \; du$ and $h_L$ spans the eigenspace corresponding
to $\lambda_1=-1$, we conclude that 
\begin{align*}0&=3 \mu(f,L,M)=\int h_{L}D_{\mu,M}(f) \; du \\
 &= -\sum_{k=1}^\infty \lambda_k \int h_{L}f_k \frac{D_{\mu,M}(L)}{h_L}\; du=  \int  f_{1}  D_{\mu,M}(L) \;du.
\end{align*}
Since $f_{1}$ is a multiple of $h_{L}$, this implies $f_{1}=0$.
Hence
\begin{align*}3 \mu(f,f,M)&=\int f D_{\mu,M}(f) \;du= -\sum_{k=3}^\infty \lambda_k\int f f_k \frac{D_{\mu,M}(L)}{h_L}\;du \\
 & =  -\sum_{k=3}^\infty \lambda_k\int f_k^2  \frac{D_{\mu,M}(L)}{h_L}  \;du \leq  0.
\end{align*}
Equality holds if and only if 
$f_k=0$ for $k\geq 3$. Hence $\mu(f,f,M)=0$ if and only if $f$ is the restriction of a linear functional to the unit sphere. 

The case that $\mu$ belongs to the convex cone generated by $\varphi_\theta$ with $\theta$ satisfying \eqref{eq:theta2_open} is proved along the same lines, the only
change is that instead of the eigenvalue problem \eqref{eq:elliptic_eigenvalue}, one has to consider now the eigenvalue problem
$$D_{\mu}( f)+\lambda\frac{D_{\mu} (L)}{h_{L}}f=0.$$

\section{Equality cases}

We say that the unitary valuation $\mu\in \Val_k^{U(n)}$ satisfies the Aleksandrov-Fenchel inequality if
$$\mu(f, L, M_1,\ldots, M_{k-2})^2 \geq \mu(f, f, M_1,\ldots, M_{k-2}) \mu(L, L, M_1,\ldots, M_{k-2})$$
for all convex bodies $L$, $M_1,\ldots, M_{k-2}$, and all differences of support functions $f$. In the following we will use the abbreviations $\calM=(M_1,\ldots, M_{k-2})$ and
$\mu(f, L, \calM)=\mu(f, L, M_1,\ldots, M_{k-2})$.

\begin{lemma} Suppose $\mu$ satisfies the Aleksandrov-Fenchel inequality. Let $L$ and $M_1,\ldots, M_{k-2}$ be convex bodies, $f$ the difference of two support functions and assume  
$$\mu(L,L, \calM)>0.$$
Then equality holds in the inequality
$$\mu(f,L,\calM)^2 \geq \mu(f,f,\calM)\mu(L,L ,\calM)$$
if and only if 
$$ S_\mu(f,\calM)= \alpha S_\mu(L,\calM)$$
for  some constant $\alpha$.
\end{lemma}
\begin{proof} Since $\mu(L,L, \calM)>0$ and $\mu$ satisfies the Aleksandrov-Fenchel inequality, we immediately obtain that for every $f$ the condition
$$\mu(f,L,\calM)=0$$
implies 
$$\mu(f,f,\calM)\leq 0.$$
Assume $\mu(f,L,\calM)^2 = \mu(f,f,\calM)\mu(L,L ,\calM)$ for some $f$. Then $f'= f- \lambda h_L$ with $\lambda= \mu(f,L, \calM) /\mu(L,L, \calM)$ satisfies 
$$\mu(f',L, \calM)=0 \qquad\text{and}\qquad \mu(f',f', \calM) =0.$$
Consequently, $f'$ maximizes  $\mu(f,f,\calM)$ under the constraint $\mu(f,L,\calM)=0$ and therefore there exists a constant $\alpha$ such that
$$ \alpha \mu(Z, L,\calM) = \mu(Z, f',\calM)$$
for every difference of support functions $Z$. Hence, by the definition of the mixed area measure $S_\mu$, 
$$\alpha  \int Z \; d S_\mu(L,\calM) =  \int Z \; d S_\mu(f',\calM) $$
for every $Z$ and as such $ \alpha  S_\mu(L,\calM) =  S_\mu(f',\calM)$.

If $ S_\mu(f,\calM)= \alpha S_\mu(L,\calM)$ for  some constant $\alpha$, then multiplying this identity by $f$ and $h_L$ and integrating, yields $\alpha \mu(f,L,\calM) =\mu(f,f,\calM)$ and 
$\alpha \mu(L,L,\calM) =\mu(L,f,\calM)$. Thus equality holds in the Aleksandrov-Fenchel inequality. 
\end{proof}

On a smooth manifold $M$, the H\"older space $C^{k,\alpha}$, $0<\alpha< 1$, is defined as the subspace of $C^k(M)$ such that for every coordinate neighborhood $U$ 
of $M$ the $k$-th order derivatives of the restriction $f|_U$ are locally 
H\"older continuous with exponent $0< \alpha< 1$. We say that a convex body $K$ in $\RR^n$ is  $C^{2,\alpha}_+$ if the support function of $K$ is $C^{2,\alpha}(S^{n-1})$ and 
$$\det( \overline \nabla^2h_K + h_K \overline g)>0.$$
In particular, $K$ has a $C^2$ boundary and all its principal curvatures strictly positive. 

\begin{lemma} \label{lem:decomp} Suppose  $\mu$ belongs to the convex cone generated by $\psi_\theta$ with $\theta$ satisfying \eqref{eq:theta3_open} and $M\in C^{2,\alpha}_+$.  Then 
$D_{\mu,M}\colon C^{2,\alpha}(S^{2n-1}) \to C^{0,\alpha}(S^{2n-1})$ satisfies
$$C^{0,\alpha}(S^{2n-1}) = \ker D_{\mu,M}  \oplus \operatorname{im}  D_{\mu,M},$$
where the summands are orthogonal with respect to the standard $L^2$ inner product and $\ker D_{\mu,M}$ consists precisely of the restriction of linear functionals to the unit sphere.
\end{lemma}
\begin{proof}
The assertion that $\ker D_{\mu,M}$ consists precisely of the restriction of linear functionals to the unit sphere can be proved as in Proposition~\ref{prop:neg_eigenvalues}. 
If the support function of $M$ is 
$C^\infty$ and $D_{\mu,M}\colon C^{\infty} \to C^{\infty}$, then the decomposition 
$$C^\infty(S^{2n-1}) = \ker D_{\mu,M} \oplus \operatorname{im} D_{\mu,M}$$
follows from the general theory of self-adjoint, elliptic linear differential operators, see, e.g., \cite{wells08}*{Theorem 4.12}. 
Now we may proceed exactly as in \cite{zhang94}*{Lemma~6.1}, approximating  $M$ by smooth convex bodies and using the Schauder interior estimates,
to obtain the corresponding decomposition if $M$ is only $C^{2,\alpha}_+$.

\end{proof}

\begin{theorem}
Suppose  $\mu$ belongs to the convex cone generated by $\psi_\theta$ with $\theta$ satisfying \eqref{eq:theta3_open} and $M\in C^{2,\alpha}_+$. If 
$$\mu(L,L,M)>0$$
then equality
holds in the inequality
$$\mu(f,L,M)^2 \geq \mu(f,f,M)\mu(L,L ,M)$$
if and only if there exists a constant $\alpha$ such that $\alpha h_L$ and $f$ differ by the restriction of a linear functional to the unit sphere.
If $M$ is a ball, then the above characterization extends to $\cos^2 \theta=0$ and $\frac{3(n+1)}{5n-1}$.

If $\mu$ belongs to the convex cone generated by $\varphi_\theta$ with $\theta$ satisfying \eqref{eq:theta2_halfopen} and 
$$\mu(L,L)>0$$
then equality
holds in the inequality
$$\mu(f,L)^2 \geq \mu(f,f)\mu(L,L)$$
if and only if there exists a constant $\alpha$ such that $\alpha h_L$ and $f$ differ by the restriction of a linear functional to the unit sphere.
If $\mu= \varphi_{\theta}$ with $\cos^2\theta= \frac{n+1}{2n}$, then  equality holds if and only if
there exists a constant $\alpha$ such that $\alpha h_L$ and $f$ differ by an element of $\calH_{1,0}\oplus \calH_{0,1} \oplus \calH_{1,1}$. 

\end{theorem}

\begin{proof} Let $Z$ be a convex body. Multiplying the equality $S_\mu(f,M)=\alpha S_\mu(L,M)$
by the support function of $Z$, integrating and using \eqref{eq:linearityS} and \eqref{eq:symmetryS}, we obtain
$$\int (f-\alpha h_L) \; dS_\mu(M,Z) =0.$$
Consequently, 
 $$\int (f-\alpha h_L) D_\mu(M,g) =0$$
for every $g\in C^{2,\alpha}$ and hence, by Lemma~\ref{lem:decomp},  $f$ and $\alpha h_L$ differ only by the restriction of a linear functional to the unit sphere.

Using Proposition~\ref{prop:tildeD} instead of Lemma~\ref{lem:decomp}, the remaining cases can be proved. 
\end{proof}

Now we show that the bound \eqref{eq:theta2} is optimal.

\begin{proposition}\label{prop:counterexample} If $\mu=\varphi_\theta$ with
$$\frac{n+1}{2n}< \cos^2\theta,$$
then there exist convex bodies $K,L$ such that 
$$\mu(K,L)^2 < \mu(K,K)\mu(L,L).$$
\end{proposition}

\begin{proof}
Let $L=B$ be the unit ball in $\CC^n$ and $f\in \calH_{1,1}$ be real-valued and non-zero (e.g., $f(z)=\operatorname{Re}(z_1\overline{z_2})$). For $\varepsilon$ sufficiently small $1+ \varepsilon f$ is the support function of a convex body $K$.  
Since $1$ and $f$ are eigenfunctions of \eqref{eq:Dmu2}, we obtain
$$D_\mu(K)=\frac{1}{\omega_{2n-2}}\left((2(n-1)c_{0}+c_{1})+((4n+1)c_{1}-2(3n+1)c_{0}) \varepsilon f\right).$$
Since $1$ and $f$ are orthogonal with respect to the standard $L^2$ inner product, we have
\begin{align*}
\mu(K,L)& =\mu(L,L)=  2\pi (2(n-1)c_{0}+c_{1}), 
\\\mu(K,K)&=  2\pi    \left( 2(n-1)c_{0}+c_{1}+\frac{((4n+1)c_{1}-2(3n+1)c_{0})\varepsilon^2}{2n\omega_{2n}}\int f^2\; du\right).
\end{align*}
Since $(4n+1)c_{1}-2(3n+1)c_{0}>0$, we obtain
$$\mu(K,L)^2 < \mu(K,K)\mu(L,L).$$

\end{proof}

\section{Brunn-Minkowski and isoperimetric inequalities} 

A straightforward consequence of the Aleksandrov-Fechel inequality \eqref{eq:aleksandrov_fenchel} is the following generalization of the Brunn-Minkowski inequality:
For $m\in \{2,\ldots, n\}$ and all convex bodies $K_0,K_1,K_{m+1},\ldots, K_n$ in $\RR^n$,
\begin{align}&V(K_0+ K_1 [m], K_{m+1},\ldots, K_n)^\frac{1}{m} \label{eq:general_brunn-minkowski}\\
 & \qquad \qquad  \geq V(K_0 [m], K_{m+1},\ldots, K_n)^\frac{1}{m} + V(K_1 [m], K_{m+1},\ldots, K_n)^\frac{1}{m}, \notag
\end{align}
where here and in the following we use the shorthand 
$$(K [m], K_{m+1},\ldots, K_n)= (\underbrace{K,\ldots, K}_{m \text{ times}}, K_{m+1}, \ldots, K_n).$$
Proofs of \eqref{eq:general_brunn-minkowski} were first published by Fenchel \cite{fenchel36b} and Aleksandrov \cite{aleksandrov37}. In the case $m=n$ the inequality \eqref{eq:general_brunn-minkowski} reduces to 
the classical Brunn-Minkowski inequality. 

\begin{theorem}\label{thm:BM}
Suppose $\mu$ belongs to the convex cone generated by the valuations $\psi_\theta$ 
with $\theta$ satisfying \eqref{eq:theta3} and  $m\in \{2,3\}$. Then 
\begin{align}&\mu(K_0+ K_1 [m], K_{m+1},\ldots, K_3)^\frac{1}{m} \label{eq:hermitianBM}\\
 & \qquad \qquad  \geq \mu(K_0 [m], K_{m+1},\ldots, K_3)^\frac{1}{m} + \mu(K_1 [m], K_{m+1},\ldots, K_3)^\frac{1}{m}, \notag
\end{align}
for all convex bodies $K_0,K_1,K_{m+1},\ldots, K_3$  in $\CC^n$. If $\theta$ satisfies \eqref{eq:theta3_open}, $K_3$ (or $K_1$ if $m=3$) is of class $C^{2,\alpha}_+$, and
\begin{equation}\label{eq:hermitianBM_pos}  \mu(K_1 [m], K_{m+1},\ldots, K_3)>0,\end{equation}
then equality holds in the inequality if and only if $K_0$ and $K_1$ are homothetic.

 A corresponding inequality with $m=2$ holds if $\mu$ belongs to the convex cone generated by the valuations $\varphi_\theta$ with $\theta$ satisfying \eqref{eq:theta2}.
If $\theta$ satisfies \eqref{eq:theta2_halfopen} and  $\mu(K_1,K_1)>0$,
then equality holds if and only if $K_0$ and $K_1$ are homothetic.
\end{theorem}
\begin{proof} 
In order to deduce \eqref{eq:hermitianBM} from the Aleksandrov-Fenchel inequality \eqref{eq:AF3} one may proceed exactly as in the case of \eqref{eq:general_brunn-minkowski}, see, e.g., \cite{schneider14}*{Theorem 7.4.5}.	
Turning to the equality cases, first note that \eqref{eq:hermitianBM} implies the concavity of the function
$$ f(\lambda) = \mu((1-\lambda) K_0 + \lambda K_1 [m],  K_{m+1},\ldots, K_3)^\frac{1}{m}, \qquad \lambda\in[0,1]$$
and that $f$ is $C^\infty$ on $(0,1)$ by \eqref{eq:hermitianBM_pos}. If equality holds in \eqref{eq:hermitianBM}, then
$$f(\lambda)- (1-\lambda) f(0)-\lambda f(1)\geq 0$$
attains a global minimum at $\lambda=1/2$. Since $f$ is also concave, we obtain
$$0=f''(1/2)= 4 (m-1)\mu_{(0)} ^{\frac{1}{m}-2} \left( \mu_{(0)}\mu_{(2)} - \mu_{(1)}^2\right),$$
where 
$$\mu_{(i)}= \mu(2^{-1} ( K_0 + K_1)[i], K_1[m-i], K_{m+1}, \ldots, K_3)$$
for $i=0,1,2$. 
From Theorem~\ref{th2} we deduce that $K_0$ and $K_1$ are homothetic.

\end{proof}

The term ``quermassintegral'' is derived from the German ``Querma\ss'', which can be the measure of either a cross-section or a projection. The classical isoperimetric inequalities for quermassintegrals ($k=1,\ldots,n-1$),
\begin{equation}\label{eq:quermassintegral}\left(\int_{\Grass_{k-1}} \vol_{k-1}(K|E) \; dE \right)^k \geq \frac{\omega_{k-1}^k}{\omega_k^{k-1}} \left(\int_{\Grass_{k}} \vol_{k}(K|E) \; dE \right)^{k-1}
\end{equation}
are a direct consequence of the Aleksandrov-Fenchel inequality \eqref{eq:aleksandrov_fenchel}. Applying \eqref{eq:AF3} iteratively, yields, as in the euclidean case, the inequalities
\begin{equation}\label{eq:AF_iso}
 \mu(K,L,L)^3 \geq  \mu(K)\mu(L)^2
\end{equation}
and 
$$\mu(K,L,M)^3 \geq \mu(K) \mu(L)\mu(M).$$
In particular, letting $K$ or $L$ be the unit ball of $\CC^n$, we obtain
\begin{equation}\label{eq:muquer}
 \mu(K,B,B)^3\geq \mu(B)^2 \mu(K)\qquad \text{and}\qquad \mu(K,K,B)^3\geq \mu(B) \mu(K)^2.
\end{equation}
In both inequalities, as a consequence of Theorem~\ref{th2}, if the left-hand side is non-zero, then equality holds if and only if $B$ is a ball.

\begin{lemma} \label{lem:mu_der} Let $\theta, \theta' \in [0,\pi/2]$. 
 If $\mu=\psi_\theta$, then 
$$\mu(K,K,B)= \frac{4}{3} \varphi_{\theta'}$$
with $3 \cos^2\theta' = \cos^2\theta$.
\end{lemma}
\begin{proof}
 From the definition of $S_\mu$ and \eqref{eq:RuminD}, we have for $\mu=c_0\mu_{3,0}+ c_1\mu_{3,1}$
$$\mu(K,K,B)= \frac{2\omega_{2n-2}}{3\omega_{2n-3}} \left( (c_1 +(n-2)c_0) \mu_{2,0}+ (n-1)c_1 \mu_{2,1} \right).$$
This and Lemma~\ref{lem:quermassintegrals} imply that $\mu(K,K,B)$ is a multiple of $\varphi_\theta$ if and only if
$$\frac{3+\cos^2 \theta}{\cos^2\theta-3(2n-1)}=\frac{1+\cos^2\theta'}{\cos^2\theta'-2n+1}$$
which is the case if and only if $3 \cos^2\theta' = \cos^2\theta$.
\end{proof}

Combining \eqref{eq:muquer} and Lemma \ref{lem:mu_der} yields the following.

\begin{theorem} Let $\theta$ satisfy \eqref{eq:theta3} and choose $\theta'\in [0,\pi/2]$ such that 
$3 \cos^2\theta' = \cos^2\theta$. Then 
$$\left( \int_{\Grass_2(\theta')} \vol_2(K|E) \; dE \right)^3 \geq \frac{9\pi}{16} \left( \int_{\Grass_3(\theta)} \vol_3(K|E) \; dE \right)^2$$
for all convex bodies in $\CC^n$. If the left-hand side is non-zero, then equality holds if and only if $K$ is a ball.
\end{theorem}

By \eqref{eq:dimension}, the space $\Val^{U(n)}_1$ is $1$-dimensional and as such spanned by the first intrinsic volume or mean width which is defined by
$$\int_{\Grass_1} \vol_1(K|E)\; dE.$$
In particular, $\mu(K,B)$ is a constant multiple of the mean width of $K$. Hence Theorems~\ref{th1} and \ref{th2} imply the following.

\begin{theorem}
If $\theta$ satisfies \eqref{eq:theta2}, then 
$$\left(\int_{\Grass_1} \vol_1(K|E) \; dE \right)^2 \geq \frac{4}{\pi} \int_{\Grass_2(\theta)} \vol_2(K|E) \; dE $$
for all convex bodies in $\CC^n$. If the left-hand side is non-zero, then equality holds if and only if $K$ is a ball.
\end{theorem}

\end{document}